\documentclass[a4paper]{article}
\usepackage[utf8]{inputenc}
\usepackage[english]{babel}
\usepackage{amsmath,amsthm,amssymb}
\usepackage{bbm}	%% used for \mathbbm
\usepackage[square,numbers]{natbib}	%% bibliography
\usepackage{graphicx, color, soul} %% soul for highlight
\usepackage[title]{appendix}
\usepackage[paperwidth=19.5cm,paperheight=25cm]{geometry}
\usepackage{url}
\usepackage{blkarray} %matrix with labels
\usepackage[colorlinks=true,linkcolor=blue,citecolor=blue]{hyperref}
\usepackage[capitalize,nameinlink]{cleveref}

\usepackage{authblk} %% authors and affiliation

\usepackage[shortlabels]{enumitem}	%% change style of lists as "[(i)]"
\setlist{labelsep=.25in,leftmargin=*,labelindent=1cm,topsep=2pt,noitemsep}	%% adjust lists
\setlist[enumerate]{label=(\roman*)}

\theoremstyle{plain}
\newtheorem{theorem}{Theorem}[section]

\newtheorem{proposition}[theorem]{Proposition}
\newtheorem{lemma}[theorem]{Lemma}

\newtheorem*{theorem*}{Theorem}
\theoremstyle{remark}

\newtheorem*{lemma*}{Lemma}
\theoremstyle{definition}
\newtheorem{example}[theorem]{Example}

\numberwithin{equation}{section}

\newcommand{\Polya}{P\'{o}lya }
\newcommand{\limN}{\lim_{n\rightarrow\infty}}

\newcommand{\limsupN}{\limsup_{n\rightarrow\infty}}
\newcommand{\liminfN}{\liminf_{n\rightarrow\infty}}
\newcommand{\sumN}{\sum_{i=1}^n}
\newcommand{\sumINF}{\sum_{n=1}^\infty}
\usepackage{geometry}
\title{Infinite-color randomly reinforced urns with dominant colors}
\author{Hristo Sariev\textsuperscript{ a}\thanks{Corresponding author: h.sariev@math.bas.bg}, Sandra Fortini\textsuperscript{ b}, Sonia Petrone\textsuperscript{ b}\vspace{.2cm}\\
\textit{\textsuperscript{a }Institute of Mathematics and Informatics, Bulgarian Academy of Sciences}\\
\textit{\textsuperscript{b }Department of Decision Sciences, Universit\`a Commerciale Luigi Bocconi}}
\date{}

\begin{document}

\maketitle
%\bigskip

%\begin{aug}
%\author[bas]{\fnms{Hristo} \snm{Sariev} \ead[label=e1,mark]{h.sariev@math.bas.bg}},
%\author[bocc]{\fnms{Sandra} \snm{Fortini} \ead[label=e2,mark]{sandra.fortini@unibocconi.it}}
%\and
%\author[bocc]{\fnms{Sonia} \snm{Petrone}\ead[label=e3,mark]{sonia.petrone@unibocconi.it}}

%\address[bas]{Institute of Mathematics and Informatics, Bulgarian Academy of Sciences, \lq\lq Akad. Georgi Bonchev" bl. 8, Sofia 1113, Bulgaria. \printead{e1}}
%\address[bocc]{Department of Decision Sciences, Bocconi University, via Roentgen 1, 20136 Milano, Italy. \printead{e2,e3}}
%\end{aug}

\begin{abstract}
\noindent We define and prove limit results for a class of {\em dominant \Polya sequences}, which are randomly reinforced urn processes with {\em color-specific} random weights and {\em unbounded} number of possible colors. Under fairly mild assumptions on the expected reinforcement, we show that the predictive and the empirical distributions converge almost surely (a.s.) in total variation to the same random probability measure $\tilde{P}$; moreover, $\tilde{P}(\mathcal{D})=1$ a.s., where $\mathcal{D}$ denotes the set of {\em dominant} colors for which the expected reinforcement is maximum. In the general case, the predictive probabilities and the empirical frequencies of any $\delta$-neighborhood of $\mathcal{D}$ converge a.s. to one. That is, although {\em non}-dominant colors continue to be regularly observed, their distance to $\mathcal{D}$ converges in probability to zero. We refine the above results with rates of convergence. We further hint potential applications of dominant \Polya sequences in randomized clinical trials and species sampling, and use our central limit results for Bayesian inference.
\end{abstract}



\section{Introduction}
\label{section:intro}

Random processes with reinforcement constitute an important class of mathematical models that are of interest to probabilistis and practitioners alike; see \cite{pemantle2007} for a review on the subject. A classic example is the $k$-color {\em \Polya urn} model, which describes the composition of an urn with balls of $k$ different colors that is sequentially being sampled and reinforced with an additional ball of the observed color. Denote by $X_n\in\{1,\ldots,k\}$ the color of the additional ball at step $n$. Then, $X_1$ is chosen from a discrete distribution with $\mathbb{P}(X_1=j)=m_j\bigl/\sum_{l=1}^km_l$, $j=1,\ldots,k$, where $m_j>0$ is the initial number of balls of color $j$ in the urn, and, for every $n\geq1$,
$$\mathbb{P}(X_{n+1}=j|X_1,\ldots,X_n)=\frac{m_j+\sumN\delta_{X_i}(\{j\})}{\sum_{l=1}^km_l+n},$$
where $\delta_x$ is the Dirac measure at the point $x$. It is well-known (see, e.g., \citep[][p.30]{pemantle2007}) that, for each $j=1,\ldots,k$, as $n$ approaches infinity, the proportion of $j$-colored balls in the urn converges almost surely (a.s.) to a random variable $\tilde{p}_j$, that is,
$$\mathbb{P}(X_{n+1}=j|X_1,\ldots,X_n)\overset{a.s.}{\longrightarrow}\tilde{p}_j.$$
Moreover, the random vector $\tilde{p}=(\tilde{p}_1,\ldots,\tilde{p}_k)$ has a Dirichlet distribution with parameters $(m_j)_{j=1}^k$ and the process $(X_n)_{n\geq1}$ is exchangeable; thus, by de Finetti representation theorem for exchangeable sequences, see, e.g., \citep[][Theorem 3.1]{aldous1985}, given $\tilde{p}$, the $X_n$ are conditionally independent and identically distributed (i.i.d.) according to $\tilde{p}$.

In this paper, we study the generalization of the above urn sampling scheme towards infinite colors and random reinforcement. Let $(\mathbb{X},\mathcal{X})$ be a measurable space, $\nu$ a probability measure on $\mathbb{X}$, and $\theta>0$ a positive constant. Formally, we consider a sequence of $\mathbb{X}$-valued random variables $(X_n)_{n\geq 1}$ with predictive structure given by $\mathbb{P}(X_1\in\cdot)=\nu(\cdot)$ and, for every $n\geq 1$,
\begin{equation}\label{eq:intro:rrps}
\mathbb{P}(X_{n+1}\in\cdot\mid X_1,W_1,\ldots,X_n,W_n)=\sumN\frac{W_i}{\theta+\sum_{j=1}^nW_j}\delta_{X_i}(\cdot)+\frac{\theta}{\theta+\sum_{j=1}^nW_j}\nu(\cdot),
\end{equation}
where $(W_n)_{n\geq1}$ is a sequence of non-negative random weights. The important extension of the $k$-color \Polya urn model to infinite colors, also known as \textit{\Polya sequence}, was proposed by Blackwell and MacQueen \cite{blackwell1973} and corresponds to the case of constant reinforcement $W_n=1$ and $\nu$ possibly diffuse. By Theorem 1 in \cite{blackwell1973}, the predictive distributions \eqref{eq:intro:rrps} of a \Polya sequence converge weakly with probability one to a random probability measure $\tilde{P}$ on $\mathbb{X}$, which has a Dirichlet process distribution with parameters $(\theta,\nu)$. Moreover, given $\tilde{P}$, the random variables $X_1,X_2,\ldots$ are i.i.d. with marginal distribution $\tilde{P}$; thus, they are exchangeable. We recall that, for any exchangeable sequence, the predictive and the empirical distributions converge to the same random probability measure, which is called the \textit{directing random measure} of the process; see \citep[][Section 2]{aldous1985}.

Because of their simple predictive structure, \Polya sequences play a fundamental role in the construction of more complex procedures for nonparametric Bayesian inference; see \cite{fortini2012} for a review. On the other hand, there is a growing interest from different areas of research, such as signalling theory \citep{alexander2012}, economics \citep{beggs2005}, clinical trial design \citep{rosenberger2016}, in extending the basic \Polya urn model along the lines of time-dependent or random reinforcement. The predictive construction \eqref{eq:intro:rrps} thus represents a natural and important contribution in this direction. However, the weighing scheme in \eqref{eq:intro:rrps} may lead to an unbalanced design, which can violate the condition of exchangeability. In general, an exchangeable process $(X_n)_{n\geq1}$ with prediction rule $\eqref{eq:intro:rrps}$ is a member of the family of \textit{species sampling sequences} \cite{pitman1996}, which implies by \citep[][Theorem 1]{hansen2000} a rather restrictive form of dependence between $X_n$ and $W_n$.

Bassetti and coauthors \cite{bassetti2010} define a class of \textit{generalized} species sampling sequences that have predictive distributions of the kind $\eqref{eq:intro:rrps}$, and the process they generate is \textit{conditionally identically distributed} (c.i.d.). By \citep{berti2004}, a stochastic process $(Y_n)_{n\geq1}$ is c.i.d. if, for every $n\geq1$, given $Y_1,\ldots,Y_n$, the random variables $Y_{n+1},Y_{n+2},\ldots$ are identically distributed. Therefore, any exchangeable process is c.i.d., while the converse is generally not true; indeed, by \citep[][Proposition 2.1]{kallenberg1988}, a sequence of random variables is exchangeable if and only if it is both stationary and c.i.d. Nevertheless, c.i.d. processes are asymptotically exchangeable and thus preserve basic limit properties of exchangeable sequences; see \citep{berti2004} for more details. In \eqref{eq:intro:rrps}, the presence of the random weights $W_n$ can lead to a temporary disequilibrium in the observation process. Then it may be reasonable to assume that $(X_n)_{n\geq1}$ is c.i.d., and a sufficient condition due to \citep{bassetti2010} is that $W_n$ and $X_n$ are conditionally independent given $(X_1,W_1,\ldots,X_{n-1},W_{n-1})$.
 
In this work, we focus on the case where the random reinforcement $W_n$ in \eqref{eq:intro:rrps} \textit{depends} explicitly on the color $X_n$, and the process $(X_n)_{n\geq 1}$ is in general neither exchangeable nor c.i.d. This situation occurs in many contexts; for example, in problems of species sampling, color-specific random weights describe different levels of species adaptability. Then we expect that colors which have a higher expected reinforcement will tend to {\em dominate} the observation process, just as better-adapted species have a greater chance at survival. To the best of our knowledge, existing studies consider only the model with $k$ colors, which is known as a randomly reinforced urn (RRU), and show that the probability of drawing a color whose expected weight is maximum goes indeed to one, see, e.g., \cite{berti2010,may2009,muliere2006,zhang2014}. We thus explore the extension of RRUs to an unbounded set of possible colors, which parallels the extension provided by exchangeable \Polya sequences for $k$-color \Polya urns. 

Recently, \cite{bandyopadhyay2017,janson2019,mailler2017} introduced a general class of random processes with reinforcement, called measure-valued \Polya urn processes, which formalize the idea of urn composition as a random finite measure. In this setting, the conditional distribution \eqref{eq:intro:rrps} can be regarded as the (normalized) composition of an urn, which is ``reinforced'' at each time $n$ with the random measure $W_n\delta_{X_n}$; see \citep{fortini2021} for more details. A proper urn model to illustrate the sampling procedure implied by \eqref{eq:intro:rrps} is a weighted version of Hoppe's urn \citep{hoppe1984}. Imagine an urn initially containing $\theta$ black balls. Then, at each step $n$, a ball is picked at random from the urn. If the ball is black, we generate a new color from $\nu$ and update the urn with a random number of balls of the new color; if a non-black ball is picked instead, it is replaced together with a random number of balls of the same color. It follows that the parameter $\theta$ controls the probability of observing a new color, and $\nu$ is the color distribution. The above sampling scheme can be visualized as a randomly reinforced version of the Chinese restaurant process, which models the random partition generated by a \Polya sequence.

We study the limit behavior of the proposed class of randomly reinforced \Polya sequences under minimal assumptions on the reinforcement. Let $\bar{w}$ be the sup of the {\em expected} reinforcement as it varies across colors, and $\mathcal{D}\subseteq\mathbb{X}$ the set of {\em dominant} colors whose expected weight is $\bar{w}$. Let us denote the predictive and the empirical distributions of $(X_n)_{n\geq 1}$ by
$$P_n(\cdot)\equiv\mathbb{P}(X_{n+1}\in \cdot\mid X_1,W_1,\ldots,X_n,W_n)\qquad\mbox{and}\qquad\hat{P}_n(\cdot)\equiv\frac{1}{n}\sumN\delta_{X_i}(\cdot).$$
We show in \cref{results:dps:arbitrary} that, under certain conditions of compactness, the probability of observing {\em non}-dominant colors may not vanish, though their distance to $\mathcal{D}$ converges in probability to zero. If $\bar{w}$ is significantly larger than the expected reinforcement of any $x\in\mathcal{D}^c$, then there exists a random probability measure $\tilde{P}$ on $\mathbb{X}$ with $\tilde{P}(\mathcal{D})=1$ a.s. such that
\begin{equation}\label{eq:intro:totalvar}
\|P_n-\tilde{P}\|\overset{a.s.}{\longrightarrow}0\qquad\mbox{and}\qquad\|\hat{P}_n-\tilde{P}\|\overset{a.s.}{\longrightarrow}0,  
\end{equation}
where $\|\alpha-\beta\|\equiv\sup_{A\in\mathcal{X}}|\alpha(A)-\beta(A)|$ is the total variation distance between two probability measures $\alpha$, $\beta$ on $\mathbb{X}$. The latter implies, in particular, that $\tilde{P}$ is a.s. discrete on the set $\mathcal{D}$. On the other hand, one can show from \eqref{eq:intro:totalvar} that $(X_n)_{n\geq 1}$ is asymptotically exchangeable with limit directing random measure $\tilde{P}$; that is, as $n\rightarrow\infty$, $(X_{n+1},X_{n+2},\ldots)$ converges in distribution to an exchangeable sequence with directing random measure $\tilde{P}$, see \citep[][Lemma 8.2]{aldous1985}.

In addition, we provide set-wise rates of convergence for $P_n$ and $\hat{P}_n$ in the form of a central limit theorem; namely, \cref{results:clt:arbitrary:discont} gives conditions under which, for every $A\in\mathcal{X}$ such that $A\cap\mathcal{D}\neq\emptyset$,
$$\sqrt{n}\bigl(\hat{P}_n(A)-P_n(A)\bigr)\qquad\mbox{and}\qquad\sqrt{n}\bigl(P_n(A)-\tilde{P}(A)\bigr),$$
converge to non-degenerate Gaussian limits in the sense of stable convergence and a.s. conditional convergence, respectively. We illustrate the statistical use of these results by obtaining marginal asymptotic credible intervals for $\tilde{P}$ under a Bayesian approach.

The paper is organized as follows. The formal definition of the proposed \Polya sequences with dominant colors is given in Section \ref{section:model}. Their first-order convergence properties are derived in Section \ref{section:first-order}, which also includes examples of their potential use in species sampling and clinical trials. Section \ref{section:second-order} contains results on rates of convergence and central limit theorems. The proofs are postponed to Section \ref{section:proofs}.

\section{\Polya sequence with dominant colors}
\label{section:model}

Let $(\Omega,\mathcal{H},\mathbb{P})$ be a probability space, $(\mathbb{X},d)$ a complete separable metric space, endowed with its Borel $\sigma$-algebra $\mathcal{X}$, and $(X_n)_{n\geq 1}$ a sequence of $\mathbb{X}$-valued random variables on $(\Omega,\mathcal{H},\mathbb{P})$ with predictive distributions \eqref{eq:intro:rrps}. We assume that the random reinforcement in \eqref{eq:intro:rrps} is color-specific in the sense that
\begin{equation}\label{eq:def:weights}
W_n=h(X_n,U_n),
\end{equation}
where $(U_n)_{n\geq 1}$ is a sequence of i.i.d. random variables such that $U_n$ is independent of $(X_1,\ldots,X_n)$, and $h$ is a measurable function from $\mathbb{X}\times\mathbb{R}$ into $\mathbb{R}_+$.

Let us define the expected weight of color $x\in\mathbb{X}$ by
$$w(x)=\mathbb{E}[W_n|X_n=x],$$
and denote
$$\bar{w}=\sup_{x\in\mathbb{X}}w(x).$$
We further let $\mathcal{F}_0=\{\Omega,\emptyset\}$, $\mathcal{F}_n=\sigma(X_1,U_1,\ldots,X_n,U_n)$, $n\geq 1$, and
\begin{gather*}
\mathcal{D}=\{x\in\mathbb{X}:w(x)=\bar{w}\},\qquad
\mathcal{D}_\delta=\{x\in\mathbb{X}:d(x,\mathcal{D})<\delta\},
\end{gather*}
for every $\delta>0$, where $d(x,\mathcal{D})\equiv\inf\{d(x,y):y\in\mathcal{D}\}$ denotes the distance of $x$ from $\mathcal{D}$. By construction, $w$ is $\mathcal{X}$-measurable, so $\mathcal{D}\in\mathcal{X}$, and we call $\mathcal{D}$ the set of \emph{dominant} colors. We make the following assumptions:
\begin{equation}\label{eq:def:condition}
\begin{aligned}
&0\leq W_n\leq\beta\mbox{ for some constant }\beta;\\
&\bar{w}\in\mbox{supp}(\nu_w);\\
&\bar{w}>\sup_{x\in\mathcal{D}_\delta^c}w(x)\textnormal{ for every }\delta>0,
\end{aligned}
\end{equation}
where $\nu_w$ is the image measure of $\nu$ under $w$, and $\mbox{supp}(\nu_w)$ is the support of $\nu_w$ on $\mathbb{R}_+$, defined by
$$\mbox{supp}(\nu_w)=\{u\geq0:\forall\epsilon>0,\nu_w((u-\epsilon,u+\epsilon))>0\}.$$
The assumptions \eqref{eq:def:condition} essentially require that the colors $x$, for which the expected weight $w(x)$ is arbitrarily close to $\bar{w}$, lie in a compact set that we are able to sample from. If $\mathcal{D}$ is non-empty and $\nu(\mathcal{D})>0$, then $\bar{w}\in\mbox{supp}(\nu_w)$ is automatically satisfied, in which case we shall be interested in the relationship between $\bar{w}$ and the next highest value in the support of $\nu_w$,
$$\bar{w}^c=\sup\{u\geq0:u\in\mbox{supp}((\nu_{|\mathcal{D}^c})_w)\},$$
where $\nu_{|\mathcal{D}^c}(\cdot)=\nu(\cdot\,|\mathcal{D}^c)$. We call any process characterized by \eqref{eq:intro:rrps} with \eqref{eq:def:weights} a {\em dominant \Polya sequence} (DPS) if it satisfies the assumptions \eqref{eq:def:condition}.

\section{First-order convergence}
\label{section:first-order}

We first consider the DPS with a constant expected weight function, $w(x)\equiv\bar{w}$. Then all colors are dominant and one can show that $W_n$ and $X_n$ are conditionally uncorrelated given $\mathcal{F}_{n-1}$. Recall from Section \ref{section:intro} that $(X_n)_{n\geq 1}$ is c.i.d. whenever $W_n$ and $X_n$ are conditionally independent. In that case, general c.i.d. theory implies that the predictive and empirical distributions converge set-wise to the same random limit; see \citep[][Section 2]{berti2004}. Here, $(X_n)_{n\geq 1}$ is generally not c.i.d., yet $P_n$ and $\hat{P}_n$ converge to a common random probability measure, and convergence is in total variation.

\begin{proposition}
\label{results:dps:all}
Assume $\mathbb{X}=\mathcal{D}$. Then there exists a random probability measure $\tilde{P}$ on $\mathbb{X}$ 
such that
$$\|P_n-\tilde{P}\|\overset{a.s.}{\longrightarrow}0\qquad\mbox{and}\qquad\|\hat{P}_n-\tilde{P}\|\overset{a.s.}{\longrightarrow}0.$$
\end{proposition}

The next result concerns the opposite extreme case of a strictly increasing $w$ function on $\mathbb{X}=[0,1]$. Then ``$1$'' is the only dominant color and we show that the predictive and the empirical distributions converge weakly with probability one to a probability measure degenerate at ``$1$''. Even so, the probability of drawing a {\em non}-dominant color may not vanish.

\begin{proposition}
\label{results:dps:mon}
Let $\mathbb{X}=[0,1]$ and $(X_n)_{n\geq 1}$ be a DPS with a continuous and strictly increasing expected weight function $w$. Then, as $n\rightarrow\infty$, $\frac{1}{n}\sumN W_i\overset{a.s.}{\longrightarrow}w(1)$ and
$$P_n\overset{w}{\longrightarrow}\delta_1\quad\mbox{a.s.}[\mathbb{P}]\qquad\mbox{and}\qquad\hat{P}_n\overset{w}{\longrightarrow}\delta_1\quad\mbox{a.s.}[\mathbb{P}].$$
Moreover, there exists a $[0,1]$-valued random variable $\eta$ such that $P_n(\{1\}^c)\overset{a.s.}{\longrightarrow}\eta$.
\end{proposition}

The sampling scheme implied by a strictly increasing $w$ can be described as follows: colors are picked and reinforced until a new color closer to ``$1$'' is drawn. This color is then increasingly preferred by the model until a better alternative appears, which happens almost surely. In fact, as we will show in \cref{results:clt:arbitrary:rate}, \textit{non}-dominant colors are forgotten quicker the more distant they are from ``$1$''.

Given \cref{results:dps:all,results:dps:mon}, we are ready to state our main result of this section, which involves a DPS with an arbitrary expected weight function $w(\cdot)$ and a general space of colors $\mathbb{X}$. We show that, as $n$ increases, the observations tend to concentrate around the set of dominant colors. If further $\bar{w}>\bar{w}^c$, then the probability of observing {\em non}-dominant colors converges a.s. to zero.

\begin{theorem}\label{results:dps:arbitrary}
For any DPS $(X_n)_{n\geq 1}$, as $n\rightarrow\infty$, $\frac{1}{n}\sumN W_i\overset{a.s.}{\longrightarrow}\bar{w}$ and
$$P_n(\mathcal{D}_\delta^c)\overset{a.s.}{\longrightarrow}0\qquad\mbox{for every }\delta>0.$$
Moreover, there exists a $[0,1]$-valued random variable $\eta$ such that
$$P_n(\mathcal{D}^c)\overset{a.s.}{\longrightarrow}\eta.$$
If, in addition, it holds that $\bar{w}>\bar{w}^c$, then $\eta=0$ and there exists a random probability measure $\tilde{P}$ on $\mathbb{X}$ with $\tilde{P}(\mathcal{D})=1$ a.s.$[\mathbb{P}]$ such that
$$\|P_n-\tilde{P}\|\overset{a.s.}{\longrightarrow}0\qquad\mbox{and}\qquad\|\hat{P}_n-\tilde{P}\|\overset{a.s.}{\longrightarrow}0.$$
\end{theorem}

If $\bar{w}=\bar{w}^c$, then $P_n(\mathcal{D}_\delta)\overset{a.s.}{\longrightarrow} 1$  and, by \cref{results:dps:arbitrary}, $\liminf_n P_n(\mathcal{D}_\delta\backslash\mathcal{D})$ can be strictly positive. In that case, \emph{non}-dominant colors continue to be regularly observed yet get closer and closer to $\mathcal{D}$. In fact, using dominated convergence,
\begin{equation}\label{eq:first-order:probability}
d(X_n,\mathcal{D})\overset{p}{\longrightarrow}0.
\end{equation}

If $\bar{w}>\bar{w}^c$, then, in particular, $P_n\overset{w}{\longrightarrow}\tilde{P}$ a.s.$[\mathbb{P}]$ and $\hat{P}_n\overset{w}{\longrightarrow}\tilde{P}$ a.s.$[\mathbb{P}]$. It follows from a generalized martingale convergence theorem (see \citep[][Theorem 2]{blackwell1962}) that the conditional distribution of $X_{n+1}$ given $(X_1,\ldots,X_n)$ converges weakly with probability one to $\tilde{P}$. Thus, by Lemma 8.2 in \citep{aldous1985}, $(X_n)_{n\geq 1}$ is asymptotically exchangeable with limit directing random measure $\tilde{P}$.

The reinforced urn scheme \eqref{eq:intro:rrps} implies that there will be ties in a sample $(X_1, \ldots, X_n)$ of size $n$ with positive probability. Let us denote the number of distinct colors in $(X_1, \ldots, X_n)$ by 
$$L_n=\max\bigl\{k\in\{1,\ldots,n\}:X_k\notin\{X_1,\ldots,X_{k-1}\}\bigr\}.$$
The next result shows that $L_n$ is approximately $\bar{w}^{-1}\theta\log n$ for large $n$.

\begin{proposition}\label{results:dps:arbitrary:number}
For any DPS $(X_n)_{n\geq 1}$, if $\nu$ is diffuse, then, as $n\rightarrow\infty$,
$$\frac{L_n}{\log n}\overset{a.s.}{\longrightarrow}\frac{\theta}{\bar{w}}.$$
\end{proposition}

\subsection{Examples}

We close this section with a few examples and hints of potential areas of application.

\begin{example}[$k$-color RRU]\label{example:k-color:conv}
Let $F_1,\ldots F_k$ be probability distribution functions on $[0,\beta]$. Denote $w_j=\int tdF_j(t)$, $j=1,\ldots,k$, and suppose that, for some $k_0\in\{1,\ldots,k\}$,
\begin{equation}\label{eq:example:k-color:weights}
w_1=\cdots=w_{k_0}=\bar{w}>\bar{w}^c\equiv\max_{k_0<j\leq k}w_j.
\end{equation}
The $k$-color RRU is a DPS that associates $W_n|\{X_n=j\}\sim F_j$ to color $j=1,\ldots,k$. Then $w(j)=w_j$, $j=1,\ldots,k$ and $\mathcal{D}=\{1,\ldots,k_0\}$. RRUs with $k_0=1$ and $k_0=k$ have been studied by \cite{may2009,muliere2006}, while \cite{berti2010,zhang2014} consider the case $1<k_0<k$, although \cite{berti2010} assumes \eqref{eq:example:k-color:weights} only asymptotically.

Denote by $P_n(\cdot\,|\mathcal{D})=P_n(\cdot\cap\mathcal{D})\bigl/P_n(\mathcal{D})$ the predictive distribution of $X_{n+1}$ restricted to $\mathcal{D}$. It follows from \cref{results:dps:all} that $P_n(\{j\}|\mathcal{D})\overset{a.s.}{\longrightarrow}\tilde{p}_j$, $j=1,\ldots,k_0$, for some non-negative random variables $\tilde{p}_1,\ldots,\tilde{p}_{k_0}$ such that $\sum_{j=1}^{k_0}\tilde{p}_j=1$ a.s.$[\mathbb{P}]$. As $P_n(\mathcal{D})\overset{a.s.}{\longrightarrow}1$ by \cref{results:dps:arbitrary}, then
$$P_n(\{j\})=P_n(\mathcal{D})P_n(\{j\}|\mathcal{D})+P_n(\mathcal{D}^c)P_n(\{j\}|\mathcal{D}^c)\overset{a.s.}{\longrightarrow}\tilde{p}_j.$$
Therefore, $\|P_n-\tilde{P}\|\overset{a.s.}{\longrightarrow}0$ and, similarly, $\|\hat{P}_n-\tilde{P}\|\overset{a.s.}{\longrightarrow}0$, where $\tilde{P}=\sum_{j=1}^{k_0}\tilde{p}_j\delta_{j}$.
\end{example}

\begin{example}[Unimodal DPS]\label{example:clinical-trials}
Suppose that we have a DPS with a unique dominant color, \textit{i.e.}, $\mathcal{D}=\{x_0\}$ for some $x_0\in\mathbb{X}$. It follows from \eqref{eq:first-order:probability} that
$$X_n\overset{p}{\longrightarrow}x_0.$$
Moreover,
$$P_n\overset{w}{\longrightarrow}\delta_{x_0}\qquad\mbox{a.s.}[\mathbb{P}].$$
Indeed, for every open set $G\in\mathcal{X}$, we have that, if $x_0\in G$, then there exists $\delta>0$ such that $x_0\in\mathcal{D}_\delta\subseteq G$; therefore, by \cref{results:dps:arbitrary},
$$\liminfN P_n(G)\geq\liminfN P_n(\mathcal{D}_\delta)=1=\delta_{x_0}(G)\qquad\mbox{a.s.}[\mathbb{P}].$$
As $\delta_{x_0}(G)=0$ otherwise, a.s. weak convergence of $P_n$ to $\delta_{x_0}$ follows.

Unimodal DPSs can be used in the context of response-adaptive randomization procedures for clinical trials (see \cite{rosenberger2016} for a general reference) to provide an ethical design that can deal with an {\em unbounded} number or even a {\em continuum} of potential treatments. Consider a stylized trial for determining the optimal dose of a drug. Let $X_n$ represent the dose given to the $n$th patient who enters the trial, and denote by $W_n$ the patient's response to it. Then the reinforced predictive scheme \eqref{eq:intro:rrps} with a unimodal $w(\cdot)$ function provides a response-adaptive rule for selecting the dose for the next patient. It follows from \cref{results:dps:arbitrary} that such a design is ethical in the sense that, for any $\delta >0$, $P_n(\mathcal{D}_\delta)\overset{a.s.}{\longrightarrow}1$; that is, even in situations where it is practically impossible to select $x_0$ exactly (e.g., $\nu(\{x_0\})=0$ with $\nu$ is diffuse), the probability of assigning the next patient a dose arbitrarily close to $x_0$ converges a.s. to one.
\end{example}

\begin{example}[Species sampling model with dominant species]
Exchangeable species sampling models are of interest in many fields and, in particular, in the broad area of Bayesian nonparametrics (see \citep{pitman1996}), yet forms of competition between species can easily break the symmetry imposed by exchangeability. Dominant \Polya sequences can thus provide an extension to the framework of species sampling models by addressing situations where exchangeability holds only asymptotically {\em and} a restricted set of dominant species tends to prevail. Another related direction of application is in random graph theory, where conciliating between exchangeability and graph density is a challenging issue (see \cite{caron2017}), and dominant \Polya sequences can suggest `sparse' and asymptotically exchangeable graph structures. 

Consider again the sampling scheme \eqref{eq:intro:rrps}. In this case, interest lies on the {\em categorical} species and $X_n$ plays the role of a label chosen, say, in $\mathbb{X}=[0,1]$ from a uniform $\nu$. Take $0<w_1<w_2<\infty$ and $p\in(0,1)$. Let $W_n=w(X_n)$ where
$$w(x)=w_1\cdot\mathbbm{1}_{[0,p)}(x)+w_2\cdot\mathbbm{1}_{[p,1]}(x)\quad\mbox{for }x\in\mathbb{X}.$$
Then $(X_n)_{n\geq 1}$ is a DPS with $\mathcal{D}=[p,1]$. According to this scheme, when a {\em new} species appears, it has probability $p$ to be {\em non}-dominant and $(1-p)$ to be dominant, independently over $n$. As $p>0$, then $\sumN\delta_{X_i}(\mathcal{D}^c)=\infty$ a.s.$[\mathbb{P}]$ and we discover {\em non}-dominant species infinitely often. However, $\hat{P}_n([p,1])\overset{a.s.}{\longrightarrow}1$ by \cref{results:dps:arbitrary}, so the discovery rate of {\em non}-dominant species is of order less than $n$ (see also \cref{results:clt:arbitrary:rate}).
\end{example}

\section{Rates and central limit results}\label{section:second-order}

Let $(X_n)_{n\geq 1}$ be a DPS such that $\bar{w}>\bar{w}^c$, and let $\tilde{P}$ be a random probability measure as in \cref{results:dps:arbitrary}. It follows for every $A\in\mathcal{X}$ that
$$P_n(A)-\tilde P(A)\overset{a.s.}{\longrightarrow}0\qquad\mbox{and}\qquad\hat{P}_n(A)-P_n(A)\overset{a.s.}{\longrightarrow}0.$$
We proceed to study the rate of convergence by first considering $A=\mathcal{D}^c$.

\subsection{Rate of convergence for dominated colors}

Assume $\bar{w}>\bar{w}^c$. By \cref{results:dps:arbitrary}, both $\hat{P}_n(\mathcal{D}^c)$ and $P_n(\mathcal{D}^c)$ converge a.s. to zero. The next result shows that $\hat{P}_n(\mathcal{D}^c)$ and $P_n(\mathcal D^c)$ are $O(n^{-(\bar{w}-\bar{w}^c)/\bar{w}})$, as $n\rightarrow\infty$.

\begin{proposition}\label{results:clt:arbitrary:rate}
Suppose $\bar{w}>\bar{w}^c$. Then there exists a random variable $\xi\in[0,\infty)$ such that, letting $\gamma=1-\bar{w}^c/\bar{w}$, as $n\rightarrow\infty$,
$$n^\gamma\cdot\hat{P}_n(\mathcal{D}^c)\overset{a.s.}{\longrightarrow}\xi\qquad\mbox{and}\qquad n^\gamma\cdot P_n(\mathcal{D}^c)\overset{a.s.}{\longrightarrow}\frac{\bar{w}^c}{\bar{w}}\xi.$$
Moreover, $n^\gamma\cdot\hat{P}_n(\mathcal{D}^c)\overset{a.s.}{\longrightarrow}\infty$ and $n^\gamma \cdot P_n(\mathcal{D}^c)\overset{a.s.}{\longrightarrow}\infty$ when $\gamma>1-\bar{w}^c/\bar{w}$.
\end{proposition}

\subsection{Central limit theorem}

Let $A\in\mathcal{X}$ be such that $A\cap\mathcal{D}\neq\emptyset$. We provide conditions under which
$$C_n(A)=\sqrt{n}\bigl(\hat{P}_n(A)-P_n(A)\bigr)\qquad\textnormal{and}\qquad D_n(A)=\sqrt{n}\bigl(P_n(A)-\tilde{P}(A)\bigr),$$
converge to non-degenerate Gaussian limits.  The results are given in terms of stable and almost sure (a.s.) conditional convergence, which we briefly recall.

{\em Almost sure conditional convergence}. Let $\mathcal{G}=(\mathcal{G}_n)_{n\geq 0}$ be a filtration on $\mathcal{H}$, and $\tilde{Q}$ a random probability measure on $\mathbb{X}$. A sequence $(Y_n)_{n\geq 1}$ is said to converge to $\tilde{Q}$ in the sense of a.s. conditional convergence with respect to (w.r.t.) $\mathcal{G}$ if, as $n\rightarrow\infty$,
$$\mathbb{P}(Y_n\in\cdot\mid\mathcal{G}_n)\overset{w}{\longrightarrow}\tilde{Q}(\cdot)\quad\mbox{a.s.}[\mathbb{P}].$$
We refer to \cite{crimaldi2009} for more details.

{\em Stable convergence}. Stable convergence is a strong form of convergence in distribution, albeit weaker than a.s. conditional convergence. A sequence $(Y_n)_{n\geq 1}$ converges stably to $\tilde{Q}$ if, as $n\rightarrow\infty$,
$$\mathbb{E}\bigl[V\,f(Y_n)\bigr]\longrightarrow\mathbb{E}\Bigl[V \int_\mathbb{X}f(x)\tilde{Q}(dx)\Bigr],$$
for all continuous bounded functions $f$ and any integrable random variable $V$. The main application of stable convergence is in central limit theorems that allow for mixing variables in the limit. See \cite{hausler2015} for a complete reference on stable convergence.

In the sequel, the stable and a.s. conditional limits will be some Gaussian law, which we denote by $\mathcal{N}(\mu,\sigma^2)$ for parameters $(\mu,\sigma^2)$, where $\mathcal{N}(\mu,0)=\delta_\mu$. Recall from Section \ref{section:model} that $\mathcal{F}=(\mathcal{F}_n)_{n\geq0}$ is the filtration on $\mathcal{H}$ generated by $(X_n,U_n)_{n\geq1}$.

Lastly, we can show using standard arguments that $\mathbb{E}[f(X_{n+1})|\mathcal{F}_n]\overset{a.s.}{\longrightarrow}\int_\mathbb{X}f(x)\tilde{P}(dx)$ for every bounded measurable function $f$. Then the following limit exists with probability one
$$q_A=\limN\mathbb{E}[W_{n+1}^2\delta_{X_{n+1}}(A)|\mathcal{F}_n].$$

\begin{theorem}
\label{results:clt:arbitrary:discont}
Suppose $\bar{w}>2\bar{w}^c$. Let $A\in\mathcal{X}$ be such that $A\cap\mathcal{D}\neq\emptyset$. Define  
$$V(A)=\frac{1}{\bar{w}^2}\bigl\{(\tilde{P}(A^c))^2q_A+(\tilde{P}(A))^2q_{A^c}\bigr\}\quad\mbox{and}\quad U(A)=V(A)-\tilde{P}(A)\tilde{P}(A^c).$$
Then 
\begin{equation}\label{eq:clt:result}
C_n(A)\overset{stably}{\longrightarrow}\mathcal{N}\bigl(0,U(A)\bigr)\qquad\mbox{and}\qquad D_n(A)\overset{a.s.cond.}{\longrightarrow}\mathcal{N}\bigl(0,V(A)\bigr)\quad\mbox{w.r.t. }\mathcal{F}.
\end{equation}
\end{theorem}

It follows from \eqref{eq:clt:result} and Lemma 1 in \citep{berti2011} that
$$\sqrt{n}\bigl(\hat{P}_n(A)-\tilde{P}(A)\bigr)=C_n(A)+D_n(A)\overset{stably}{\longrightarrow}\mathcal{N}\bigl(0,U(A)+V(A)\bigr).$$
Moreover, the a.s. conditional convergence of $D_n(A)$ implies that the distribution of $\tilde{P}(A)$ has no point masses when $0<\nu(A)<1$; see the proof of Theorem 3.2 in \citep{aletti2009}. Then $0<\tilde{P}(A)<1$ a.s.$[\mathbb{P}]$, and $V(A)$ and $U(A)$ are strictly positive with probability one.

If $\mathcal{D}=\mathbb{X}$, the assumptions of \cref{results:clt:arbitrary:discont} are automatically satisfied and \eqref{eq:clt:result} follows. In the general case, conditionally on $\mathcal{D}$, we have effectively a constant expected weight function. Let us define, for $n\geq1$,
$$P_n(A|\mathcal{D})=\frac{P_n(A\cap\mathcal{D})}{P_n(\mathcal{D})}\qquad\mbox{and}\qquad\hat{P}_n(A|\mathcal{D})=\frac{\hat{P}_n(A\cap\mathcal{D})}{\hat{P}_n(\mathcal{D})}.$$
Under $\bar{w}>2\bar{w}^c$, it holds $\nu(\mathcal{D})>0$, so $P_n(\mathcal{D})>0$. Moreover, $\liminfN\hat{P}_n(\mathcal{D})>0$ a.s.$[\mathbb{P}]$ by \cref{results:dps:arbitrary}; thus, $\hat{P}_n(\mathcal{D})>0$ a.s.$[\mathbb{P}]$ for large $n$. Then, working with the sequence $(X_n)_{n\geq1}$ restricted to $\mathcal{D}$, we get
\begin{equation}\label{eq:clt:result_cond}
\begin{aligned}
&\sqrt{n}\bigl(\hat{P}_n(A|\mathcal{D})-P_n(A|\mathcal{D})\bigr)\overset{stably}{\longrightarrow}\mathcal{N}\bigl(0,U(A)\bigr),\\
&\sqrt{n}\bigl(P_n(A|\mathcal{D})-\tilde{P}(A)\bigr)\overset{a.s.cond.}{\longrightarrow}\mathcal{N}\bigl(0,V(A)\bigr)\quad\mbox{w.r.t. }\mathcal{F}.
\end{aligned}
\end{equation}
In order to show \eqref{eq:clt:result}, the assumption $\bar{w}>2\bar{w}^c$ in \cref{results:clt:arbitrary:discont} is critical. Indeed, under $\bar{w}\leq 2\bar{w}^c$, \cref{results:clt:arbitrary:rate} implies that
$$D_n(A)-\sqrt{n}\bigl(P_n(A|\mathcal{D})-\tilde{P}(A)\bigr)=\sqrt{n}\;P_n(\mathcal{D}^c)P_n(A|\mathcal{D})\overset{a.s.}{\longrightarrow}\infty;$$
thus, $D_n(A)$ (and, similarly, $C_n(A)$) fails to converge as $\sqrt{n}(P_n(A|\mathcal{D})-\tilde{P}(A))$ converges by \eqref{eq:clt:result_cond}.

\begin{example}[$k$-color RRU (Continued)]
Consider again the $k$-color RRU from \cref{example:k-color:conv}. It follows that $q_j=\lim_n\mathbb{E}[W_{n+1}^2\delta_{X_{n+1}}(\{j\})|\mathcal{F}_n]$, $j=1,\ldots,k_0$ exists a.s.$[\mathbb{P}]$. Then, arguing as in \eqref{eq:clt:result_cond},
\begin{gather*}
\sqrt{n}\bigl(\hat{P}_n(\{j\}|\mathcal{D})-P_n(\{j\}|\mathcal{D})\bigr)\overset{stably}{\longrightarrow}\mathcal{N}\Bigl(0,\frac{\tilde{p}_j}{\bar{w}^2}\Bigl\{(1-\tilde{p}_j)^2q_j+\tilde{p}_j\sum_{i\leq k_0,i\neq j}q_i\tilde{p}_i\Bigr\}-\tilde{p}_j\bigl(1-\tilde{p}_j\bigr)\Bigr),\\
\sqrt{n}\bigl(P_n(\{j\}|\mathcal{D})-\tilde{p}_j\bigr)\overset{a.s.cond.}{\longrightarrow}\mathcal{N}\Bigl(0,\frac{\tilde{p}_j}{\bar{w}^2}\Bigl\{(1-\tilde{p}_j)^2q_j+\tilde{p}_j\sum_{i\leq k_0,i\neq j}q_i\tilde{p}_i\Bigr\}\Bigr)\quad\mbox{w.r.t. }\mathcal{F}.
\end{gather*}
In addition, $\mathbb{P}(\tilde{p}_j=p)=0$, $p\in[0,1]$, so that $p_j>0$ a.s.$[\mathbb{P}]$ and $\tilde{p}_i\neq\tilde{p}_j$ a.s.$[\mathbb{P}]$, $i\neq j$.
\end{example}

\subsection{An application: asymptotic credible intervals for $\tilde{P}$}

We close with an application of \cref{results:clt:arbitrary:discont} to a problem of statistical inference on the limit structure of the process. We assume that the random weights $W_n$ are observable, as it might be the case, for example, in the context of clinical trials (see \cref{example:clinical-trials}). Suppose $\mathbb{X}=\mathbb{R}$, $\mathcal{D}=[a,b]$, $\bar{w}>2\bar{w}^c$. Define $F_n(x)=P_n((-\infty,x])$, $x\in\mathbb R$, $n\geq 1$. By \cref{results:dps:arbitrary}, there exists a random distribution function $\tilde{F}$ on $[a,b]$ such that
$$\|F_n-\tilde{F}\|\overset{a.s.}{\longrightarrow}0.$$
The process $\tilde{F}$ is measurable w.r.t. $\sigma(X_1,W_1,X_2,W_2,\ldots)$ and so cannot be calculated on the basis of a finite sample $(X_i,W_i)_{i=1}^n$. Still, one can do inference on it; under a Bayesian approach, one does so by computing the conditional distribution of $\tilde{F}$ given $(X_i,W_i)_{i=1}^n$. The latter is usually summarized by a credible band or, in practice, by plotting {\em marginal} credible intervals around $\tilde{F}(x)$ for $x$ varying on a fine grid. Take $x\in[a,b]$ such that $0<\nu((-\infty,x])<1$. The following limits exist a.s.
\begin{equation}\label{appl:second-moment}
q_1=\limN\mathbb{E}\bigl[W_{n+1}^2\delta_{X_{n+1}}([a,x])\bigr|\mathcal F_n\bigr],\qquad q_0=\limN\mathbb{E}\bigl[W_{n+1}^2\delta_{X_{n+1}}((x,b])\bigr|\mathcal F_n\bigr].
\end{equation}
It follows from \cref{results:clt:arbitrary:discont} that
$$\sqrt n\bigl(F_n(x)-\tilde{F}(x)\bigr)\overset{a.s.cond.}{\longrightarrow}\mathcal{N}(0,V_x)\quad\mbox{w.r.t. }\mathcal{F},$$
where $V_x=\bar{w}^{-2}((1-\tilde{F}(x))^2q_1+\tilde{F}(x)^2q_0)$. In order to obtain asymptotic credible intervals for $\tilde{F}(x)$, we need a convergent estimator of $V_x$. Let us define, for $n\geq 1$,
\begingroup\allowdisplaybreaks
\begin{gather*}
V_{n,x}=\frac{1}{m_n^2}\bigl((1-F_n(x))^2s_{n,1}+(F_n(x))^2s_{n,0}\bigr),\\
m_n=\frac{1}{n}\sumN W_i,\quad s_{n,1}=\frac{1}{n}\sumN W_i^2\delta_{X_i}([a,x]),\quad s_{n,0}=\frac{1}{n}\sumN W_i^2\delta_{X_i}((x,b]).
\end{gather*}
\endgroup
By \cref{results:dps:arbitrary}, $m_n\overset{a.s.}{\longrightarrow}\bar{w}$. Moreover, $\sumINF\mathbb{E}[W_n^2]/n^2\leq\sumINF \beta^2/n^2<\infty$, so Lemma 2 in \citep{berti2011} implies from \eqref{appl:second-moment} that $s_{n,1}\overset{a.s.}{\longrightarrow}q_1$ and $s_{n,0}\overset{a.s.}{\longrightarrow}q_0$. Thus,
$$V_{n,x}\overset{a.s.}{\longrightarrow}V_x.$$
A generalized Slutsky's theorem (see also \citep[][Theorem 6]{fortini2020}) implies, for every $t\in\mathbb{R}$,
$$\mathbb{P}\Bigl(\frac{\tilde{F}(x)-F_n(x)}{\sqrt{V_{n,x}/n}}\leq t\Bigl|\mathcal{F}_n\Bigr)\overset{a.s.}{\longrightarrow}{\mathcal{N}}(0,1)((-\infty,t]).$$
As a consequence, denoting by $z_\alpha$ the $(1-\alpha/2)$-quantile of the standard Normal distribution, we get
$$\liminfN\mathbb{P}\Bigl(
F_n(x)-z_\alpha\sqrt{\frac{V_{n,x}}{n}}\leq \tilde{F}(x)\leq F_n(x)+z_\alpha\sqrt{\frac{V_{n,x}}{n}}\Bigr|\mathcal{F}_n 
\Bigr)\geq 1-\alpha.$$
Thus, an asymptotic marginal credible interval for $\tilde{F}(x)$ at level $1-\alpha$ is
$$\Bigl[F_n(x)-z_\alpha\sqrt{\frac{V_{n,x}}{n}}, \; F_n(x)+z_\alpha\sqrt{\frac{V_{n,x}}{n}}\Bigr].$$

\section{Proofs}
\label{section:proofs}

We shall use the following notation throughout this section,
\begin{gather*}
N_n(A)=\theta\nu(A)+\sumN W_i\delta_{X_i}(A),\quad M_n(A)=\sumN\delta_{X_i}(A)+1,\quad N_n=\theta+\sumN W_i,
\end{gather*}
so that $P_n(A)=N_n(A)/N_n$ for every $n\geq 0$ and $A\in\mathcal{X}$, with the convention $\sum_{i=1}^0a_i=0$, and we will omit the parenthesis when $A$ is an interval for the sake of clarity. All random quantities are defined on the probability space $(\Omega,\mathcal{H},\mathbb{P})$ unless otherwise specified.

The following identity, given for any $A\in\mathcal{X}$ and $n\geq 0$, is a consequence of \eqref{eq:def:weights} and will be applied repeatedly,
\begin{equation}
\label{eq:appendix:dps:expected_weight}
\mathbb{E}[W_{n+1}\delta_{X_{n+1}}(A)|\mathcal{F}_n]=\mathbb{E}[w(X_{n+1})\delta_{X_{n+1}}(A)|\mathcal{F}_n].
\end{equation}

The proof of \cref{results:dps:all} uses the following general fact for DPSs with $\mathbb{X}=\mathcal{D}$.

\begin{lemma}
\label{results:dps:all:lemma}
Under the conditions in \cref{results:dps:all}, as $n\rightarrow\infty$,
$$\frac{n}{N_n}\longrightarrow\frac{1}{\bar{w}}\qquad\mbox{a.s.}[\mathbb{P}]\mbox{ and in }L^p\mbox{ for all }p\geq1.$$
\end{lemma}
\begin{proof}
By hypothesis, $\mathbb{E}[W_{n+1}|\mathcal{F}_n]=\bar{w}$ and $\sumINF\mathbb{E}[W_n^2]/n^2\leq\sumINF\beta^2/n^2<\infty$, so \citep[][Lemma 2]{berti2011} implies that $N_n/n\overset{a.s.}{\longrightarrow}\bar{w}$. Define $N_n^*=\sumN \{W_i-\mathbb{E}[W_i]\}$, $n\geq 1$. By a classical martingale inequality (see \citep[][Lemma 1.5]{ledoux1991} and \cite[][Lemma 3]{berti2011}), we have
$$\mathbb{P}(|N_n^*|>x)\leq 2\,\exp\bigl\{-x^2\bigr/2\beta n^2\bigr\}\qquad\mbox{for all }x>0.$$
Fix $p>0$. It follows that
$$\mathbb{E}[N_n^{-p}]=p\int_\theta^\infty\frac{1}{t^{p+1}}\mathbb{P}(N_n<t)dt\leq\frac{p}{\theta^{p+1}}\int_\theta^{\theta+\frac{n\bar{w}}{2}}\mathbb{P}(N_n<t)dt+p\int_{\theta+\frac{n\bar{w}}{2}}^\infty\frac{1}{t^{p+1}}dt.$$
Clearly, $p\int_{\theta+\frac{n\bar{w}}{2}}^\infty t^{-p-1}dt=(\theta+\frac{n\bar{w}}{2})^{-p}=O(n^{-p})$. On the other hand, for $t<\theta+\frac{n\bar{w}}{2}$,
$$\mathbb{P}(N_n<t)=\mathbb{P}(N_n^*<t-\theta-n\bar{w})\leq\mathbb{P}(|N_n^*|>\theta+n\bar{w}-t)\leq2\exp\Bigl\{-\frac{(\theta+n\bar{w}-t)^2}{2\beta^2n}\Bigr\};$$
thus, $\int_\theta^{\theta+\frac{n\bar{w}}{2}}\mathbb{P}(N_n<t)dt\leq n\bar{w}\exp\{-n\bar{w}^2/4\beta^2\}$ and, ultimately, $\mathbb{E}[N_n^{-p}]=O(n^{-p})$.
\end{proof}

\begin{proof}[\textnormal{\textbf{Proof of \cref{results:dps:all}}}]
Let $A\in\mathcal{X}$ and $n\geq 0$. Then
\begingroup\allowdisplaybreaks
\begin{align*}
P_{n+1}(A)-P_n(A)&=\frac{W_{n+1}/N_n}{1+W_{n+1}/N_n}\bigl(\delta_{X_{n+1}}(A)-P_n(A)\bigr)\\
&\leq\frac{W_{n+1}}{N_n}P_n(A^c)\delta_{X_{n+1}}(A)-\Bigl(\frac{W_{n+1}}{N_n}-\frac{W_{n+1}^2}{N_n^2}\Bigr)P_n(A)\delta_{X_{n+1}}(A^c),
\end{align*}
\endgroup
where we have used that $x-x^2\leq\frac{x}{1+x}\leq x$ for $x\geq0$; therefore,
\begin{equation}\label{eq:appendix:dps:dom_full:difference}
\mathbb{E}[P_{n+1}(A)-P_n(A)|\mathcal{F}_n]\leq \beta^2 P_n(A)P_n(A^c)N_n^{-2},
\end{equation}
and, by \cref{results:dps:all:lemma},
\begingroup\allowdisplaybreaks
\begin{gather}
\sum_{n=0}^\infty\mathbb{E}[P_{n+1}(A)-P_n(A)|\mathcal{F}_n]\leq\sum_{n=0}^\infty\frac{\beta^2}{N_n^2}<\infty\qquad\mbox{a.s.}[\mathbb{P}],\label{eq:appendix:dps:dom_full:sum}\\
\sum_{n=0}^\infty\mathbb{E}\bigl[(P_{n+1}(A)-P_n(A))^2|\mathcal{F}_n\bigr]\leq\sum_{n=0}^\infty\mathbb{E}\biggl[\frac{W_{n+1}^2}{N_n^2}\Bigr|\mathcal{F}_n\biggr]\leq\sum_{n=0}^\infty\frac{\beta^2}{N_n^2}<\infty\quad\mbox{a.s.}[\mathbb{P}].\label{eq:appendix:dps:dom_full:sum2}
\end{gather}
\endgroup
Given \eqref{eq:appendix:dps:dom_full:sum} and \eqref{eq:appendix:dps:dom_full:sum2}, there exists by \citep[][Lemma 3.2]{pemantle1999} a random variable $\tilde{p}_A$ such that
$$P_n(A)\overset{a.s.}{\longrightarrow}\tilde{p}_A.$$
Define $\tilde{Q}(A)=\tilde{p}_A$, $A\in\mathcal{X}$. Then $\tilde{Q}(\mathbb{X})=1$, $\tilde{Q}(\emptyset)=0$ and $\tilde{Q}(A)\geq0$, $A\in\mathcal{X}$, a.s.$[\mathbb{P}]$. Moreover, $\omega\mapsto\tilde{Q}(A)(\omega)$ is $\mathcal{H}$-measurable for every $A\in\mathcal{X}$, and $\tilde{Q}(A_1\cup A_2)=\tilde{Q}(A_1)+\tilde{Q}(A_2)$ a.s.$[\mathbb{P}]$ for disjoint $A_1,A_2\in\mathcal{X}$. In addition, the map $A\mapsto\mathbb{E}[\tilde{Q}(A)]$ is a probability measure on $\mathbb{X}$. Indeed, let $\{A_m\}_{m\geq 1}\subseteq\mathcal{X}$ be such that $A_k\cap A_l=\emptyset$, $k\neq l$. Put $A=\bigcup_{m=1}^\infty A_m$. It follows from  \eqref{eq:appendix:dps:dom_full:difference} that
$$\mathbb{E}[P_n(A_m)]=\sum_{i=1}^n\mathbb{E}[P_i(A_m)-P_{i-1}(A_m)]+\nu(A_m)\leq\beta^2\sum_{n=0}^\infty\mathbb{E}\bigl[P_n(A_m)N_n^{-2}]+\nu(A_m)\equiv M(m).$$
By dominated convergence theorem with respect to the counting measure, $\limN\sum_{m=1}^\infty\mathbb{E}[P_n(A_m)]=\sum_{m=1}^\infty\limN\mathbb{E}[P_n(A_m)]$, whenever $\sum_{m=1}^\infty M(m)<\infty$. The latter follows from \cref{results:dps:all:lemma} as
$$\sum_{m=1}^\infty M(m)=\beta^2\sum_{n=0}^\infty\sum_{m=1}^\infty\mathbb{E}\bigl[P_n(A_m){N_n^{-2}}\bigr]+\nu(A)\leq\beta^2\sum_{n=0}^\infty\mathbb{E}[N_n^{-2}]+\nu(A)<\infty.$$
Therefore,
\begin{equation}\label{eq:appendix:dps:dom_full:changelimits}
\mathbb{E}[\tilde{Q}(A)]=\limN\sum_{m=1}^\infty\mathbb{E}[P_n(A_m)]=\sum_{m=1}^\infty\limN\mathbb{E}[P_n(A_m)]=\sum_{m=1}^\infty\mathbb{E}[\tilde{Q}(A_m)].
\end{equation}
By \citep[][Theorem 3.1]{ghosal2017}, there exists  a random probability measure $\tilde{P}$ on $\mathbb{X}$ such that $\tilde{Q}(A)=\tilde{P}(A)$ a.s.$[\mathbb{P}]$, for every $A\in\mathcal{X}$; thus,
$$P_n(A)\overset{a.s.}{\longrightarrow}\tilde{P}(A).$$
Using standard arguments, $\mathbb{E}[f(X_{n+1})|\mathcal{F}_n]\overset{a.s.}{\longrightarrow}\int_\mathbb{X}f(x)\tilde{P}(dx)$ for every bounded measurable function $f$, which implies that $P_n\overset{w}{\longrightarrow}\tilde{P}$ a.s.$[\mathbb{P}]$ since $\mathbb{X}$ is separable.

Define $S_n=\{X_1,\ldots,X_n\}$, $n\geq 1$. Then $S_n\uparrow\{X_1,X_2,\ldots\}\equiv S$ and, by Portmanteau theorem,
$$\tilde{P}(S)=\limN\tilde{P}(S_n)\geq\limN\limsup_{k\rightarrow\infty}P_k(S_n)\geq\limN\limsup_{k\rightarrow\infty}\Bigl(1-\frac{\theta}{N_k}\Bigr)=1\qquad\textnormal{a.s.}[\mathbb{P}],$$
that is, $\tilde{P}$ is $\mathbb{P}$-a.s. discrete with random support $S$. On the other hand, proceeding as in \eqref{eq:appendix:dps:dom_full:sum} and \eqref{eq:appendix:dps:dom_full:sum2}, we can show that $P_n(\{X_m\})\overset{a.s.}{\longrightarrow}\tilde{p}_m$ for some random variable $\tilde{p}_m$, $m\geq1$. By Portmanteau theorem,
$$\tilde{P}(\{X_m\})\geq\limsupN P_n(\{X_m\})=\tilde{p}_m\qquad\textnormal{a.s.}[\mathbb{P}],\mbox{ for }m\geq 1.$$
But $1=\tilde{P}(S)=\sum_{m=1}^\infty\tilde{P}(\{X_m\})\geq\sum_{m=1}^\infty\tilde{p}_m$ a.s.$[\mathbb{P}]$ and, using the same arguments as in \eqref{eq:appendix:dps:dom_full:changelimits},
$$\mathbb{E}\Bigl[\sum_{m=1}^\infty\tilde{p}_m\Bigr]=\sum_{m=1}^\infty\limN\mathbb{E}[P_n(\{X_m\})]=\limN\mathbb{E}[P_n(S)]=1\qquad\textnormal{a.s.}[\mathbb{P}],$$
so $\sum_{m=1}^\infty\tilde{p}_m=1$ a.s.$[\mathbb{P}]$, and thus $\tilde{P}(\{X_m\})=\tilde{p}_m$ a.s.$[\mathbb{P}]$, $m\geq1$. As a result, we have
$$\mathbb{P}\bigl(P_n(\{x\})\rightarrow\tilde{P}(\{x\})\mbox{ for all }x\in S\bigr)=\mathbb{P}\bigl(P_n(\{X_m\})\rightarrow\tilde{P}(\{X_m\})\mbox{ for all }m\geq1\bigr)=1,$$
and, by Scheffe's lemma,
$$\|P_n-\tilde{P}\|\leq\sup_{x\in S}\bigl|P_n(\{x\})-\tilde{P}(\{x\})\bigr|+P_n(S^c)\leq\sum_{x\in S}\bigl|P_n(\{x\})-\tilde{P}(\{x\})\bigr|+\frac{\theta}{N_n}\nu(S^c)\overset{a.s.}{\longrightarrow}0.$$
Finally, \citep[][Lemma 2]{berti2011} implies that $\hat{P}_n(\{X_m\})\overset{a.s.}{\longrightarrow}\tilde{P}(\{X_m\})$, $m\geq 1$, and so $\|\hat{P}_n-\tilde{P}\|\overset{a.s.}{\longrightarrow}0$.
\end{proof}

The proof of \cref{results:dps:mon} makes use of the following three preliminary results.

\begin{lemma}\label{appendix:mon:lemma1} 
Under the conditions in \cref{results:dps:mon}, as $n\rightarrow\infty$, $N_n(t,1]\overset{a.s.}{\longrightarrow}\infty$ for every $t\in(0,1)$.
\end{lemma}
\begin{proof}
Fix $t\in(0,1)$. Denote $A=(t,1]$. It follows from \eqref{eq:def:condition} that $P_n(A)\geq\theta\nu(A)\bigr/(\theta+n\beta)>0$, so $\sumINF P_n(A)=\infty$, and thus $\sumINF\delta_{X_n}(A)=\infty$ a.s.$[\mathbb{P}]$ by \citep[][Theorem 1]{dubins1965}. Since $N_n(A)=\sumN(N_i(A)-N_{i-1}(A))+\theta\nu(A)$, then, by \citep[][Theorem 1]{chen1978}, $N_n(A)\overset{a.s.}{\longrightarrow}\infty$ if we can show $\sumINF\mathbb{E}[N_n(A)-N_{n-1}(A)|\mathcal{F}_{n-1}\vee\sigma(X_n)]=\infty$ a.s.$[\mathbb{P}]$. But $w$ is strictly increasing, so
$$\sumN\mathbb{E}[N_i(A)-N_{i-1}(A)|\mathcal{F}_{i-1}\vee\sigma(X_i)]=\sumN w(X_i)\delta_{X_i}(A)\geq w(t)\sumN\delta_{X_i}(A)\overset{a.s.}{\longrightarrow}\infty.$$
\end{proof}

\begin{lemma}\label{appendix:mon:lemma2}
Under the conditions in \cref{results:dps:mon}, for every $t\in(0,1)$, there exist $s\in(t,1)$ and $\lambda_t\in(0,1)$ such that, for all $\lambda\in(\lambda_t,1]$, the ratio $N_n[0,t]\bigl/N_n(s,1]^\lambda$ converges $\mathbb{P}$-a.s. to a finite limit.
\end{lemma}
\begin{proof}
Fix $t\in(0,1)$. Put $\lambda_t=(w(1)+w(t))/2w(1)$. Take $\lambda\in(\lambda_t,1]$ and $s\in(t,1)$ such that $w(s)>w(t)/\lambda_t$; such an $s$ exists as $w([0,1])$ is connected and $w(t)<w(t)/\lambda_t<w(1)$. Then, by \eqref{eq:appendix:dps:expected_weight},
\begingroup\allowdisplaybreaks
\begin{align*}
\mathbb{E}\biggl[\frac{N_{n+1}[0,t]}{N_{n+1}(s,1]^\lambda}-\frac{N_n[0,t]}{N_n(s,1]^\lambda}\Bigr|\mathcal{F}_n\biggr]&=\frac{N_n[0,t]}{N_n(s,1]^\lambda}\mathbb{E}\biggl[\frac{N_{n+1}[0,t]}{N_n[0,t]}\Bigl(1-\frac{W_{n+1}\mathbbm{1}_{\{X_{n+1}>s\}}}{N_{n+1}(s,1]}\Bigr)^\lambda-1\Bigr|\mathcal{F}_n\biggr]\\
&\leq\frac{N_n[0,t]}{N_n(s,1]^\lambda N_n}\Bigl\{w(t)-\lambda w(s)\frac{N_n(s,1]}{N_n(s,1]+\beta}\Bigr\},
\end{align*}
\endgroup
where we have used that $(1+x)^\lambda\leq1+\lambda x$ for $-1\leq x\leq 0$. Let us define, for $n,m\geq 0$,
$$R_n=w(t)-\lambda w(s)\frac{N_n(s,1]}{N_n(s,1]+\beta}\qquad\mbox{and}\qquad T_m=\inf\{n\geq m:R_n\geq 0\}.$$
Each $T_m$ is a $\mathcal{F}$-stopping time and $(N_{n\wedge T_m}[0,t]\bigr/N_{n\wedge T_m}(s,1]^\lambda)_{n\geq m}$ a non-negative supermatingale, so it converges to a random limit $L_m<\infty$, say, on $\Omega_m\in\mathcal{H}:\mathbb{P}(\Omega_m)=1$. Put $\Omega^*=\{\limsup_n R_n<0\}\cap\bigcup_{m}\Omega_m$. By \cref{appendix:mon:lemma1} and as $w(t)<\lambda w(s)$, we have $\limsup_n R_n<0$ a.s.$[\mathbb{P}]$, and thus $\mathbb{P}(\Omega^*)=1$. Take $\omega\in\Omega^*$. There exists $m_0\in\mathbb{N}$ such that $R_n(\omega)<0$, $n\geq m_0$. Then $T_{m_0}(\omega)=\infty$ and $N_n[0,t](\omega)\bigr/N_n(s,1](\omega)^\lambda\rightarrow L_{m_0}(\omega)$. Therefore, $\limN N_n[0,t]\bigr/N_n(s,1]^\lambda<\infty$ a.s.$[\mathbb{P}]$.
\end{proof}

\begin{lemma}\label{appendix:mon:lemma3}
Under the conditions in \cref{results:dps:mon}, if $\nu(t,1)>0$ for every $t\in(0,1)$, then it holds $N_n(\{1\}^c)\bigr/M_n(\{1\}^c)\overset{a.s.}{\longrightarrow}w(1).$
\end{lemma}
\begin{proof}
Fix $t\in(0,1)$. It follows from $\nu(t,1)>0$ and a slightly modified version of \cref{appendix:mon:lemma2} that $N_n[0,t)\bigl/N_n(\{1\}^c)\overset{a.s.}{\longrightarrow}0$. Then
$$\liminfN\frac{1}{P_n(\{1\}^c)}\mathbb{E}[W_{n+1}\delta_{X_{n+1}}(\{1\}^c)|\mathcal{F}_n]\geq\liminfN w(t)\frac{P_n[t,1)}{P_n(\{1\}^c)}=w(t)\qquad\mbox{a.s.}[\mathbb{P}].$$
Since $\limsupN\mathbb{E}[W_{n+1}\delta_{X_{n+1}}(\{1\}^c)|\mathcal{F}_n]\bigr/P_n(\{1\}^c)\leq w(1)$, letting $t\uparrow1$, we get
\begin{equation}\label{eq:appendix:dps:mon:subsequence}
\frac{1}{P_n(\{1\}^c)}\mathbb{E}[W_{n+1}\delta_{X_{n+1}}(\{1\}^c)|\mathcal{F}_n]\overset{a.s.}{\longrightarrow} w(1).
\end{equation}

Define $T_0=0$ and, for every $n\geq1$,
$$T_n=\inf\{m>T_{n-1}:X_m\neq 1\}.$$
As $P_n(\{1\}^c)>0$, then $M_n(\{1\}^c)\overset{a.s.}{\longrightarrow}\infty$ by \citep[][Theorem 1]{dubins1965}, and so $\mathbb{P}(X_n\neq 1\textnormal{ i.o.})=1$. Therefore, each $T_n$ is a $\mathcal{F}$-stopping time such that $T_n<\infty$ a.s.$[\mathbb{P}]$, and $T_n\overset{a.s.}{\longrightarrow}\infty$.

We further let, for $n\geq 1$,
$$X^*_n=X_{T_n},\quad W^*_n=W_{T_n},\quad P^*_n(\cdot)=\frac{P_{T_{n+1}-1}(\cdot\cap \{1\}^c)}{P_{T_{n+1}-1}(\{1\}^c)},\quad\mathcal{F}^*_n=\sigma(X^*_1,W^*_1,\ldots,X^*_n,W^*_n),$$
which are well-defined on the probability space $(\Omega^*,\mathcal{H}^*,\mathbb{P}^*)\equiv(\Omega^*,\mathcal{H}\cap\Omega^*,\mathbb{P}(\cdot\,|\Omega^*))$, where $\Omega^*\equiv\bigcap_{n=0}^\infty\{T_n<\infty\}$ and $\mathbb{P}(\Omega^*)=1$. Fix $n\in\mathbb{N}$. Take $B\in\mathcal{B}[0,1]$. Then
$$P^*_n(B)=\frac{\theta\nu(B\cap\{1\}^c)+\sum_{i=1}^{T_{n+1}-1}W_i\delta_{X_i}(B\cap\{1\}^c)}{\theta\nu(B\cap\{1\}^c)+\sum_{i=1}^{T_{n+1}-1}W_i\delta_{X_i}(\{1\}^c)}=\frac{\theta\nu(B\cap\{1\}^c)+\sum_{k=1}^nW^*_k\delta_{X^*_k}(B)}{\theta\nu(\{1\}^c)+\sum_{k=1}^nW^*_k};$$
thus, $P^*_n(B)$ is $\mathcal{F}^*_n$-measurable. Let $A\in\mathcal{F}^*_n$. Since $\mathbbm{1}_{A}=f(X^*_1,W^*_1,\ldots,X^*_n,W^*_n)$ for some measurable function $f$, we can show that
\begingroup\allowdisplaybreaks
\begin{align*}
\mathbb{P}^*\bigl(A\cap\{X^*_{n+1}\in B\}\bigr)&=\mathbb{E}\bigl[f(X_{T_1},W_{T_1},\ldots,X_{T_n},W_{T_n})\mathbbm{1}_{\{X_{T_{n+1}}\in B\}\cap\Omega^*}\bigr]\\
&=\mathbb{E}\Bigl[f(X_{T_1},W_{T_1},\ldots,X_{T_n},W_{T_n})\frac{P_{T_{n+1}-1}(B\cap\{1\}^c)}{P_{T_{n+1}-1}(\{1\}^c)}\mathbbm{1}_{\Omega^*}\Bigr]=\mathbb{E}^*[\mathbbm{1}_{A}\cdot P^*_n(B)].
\end{align*}
\endgroup
Therefore, $P^*_n$ is a version of the conditional distribution of $X^*_{n+1}$ given $\mathcal{F}^*_n$. Then, by \eqref{eq:appendix:dps:expected_weight} and \eqref{eq:appendix:dps:mon:subsequence},
$$\mathbb{P}^*\bigl(\limN\mathbb{E}^*[w(X^*_{n+1})|\mathcal{F}^*_n]=w(1)\bigr)=1.$$
Moreover, $\sumINF\mathbb{E}^*[(W^*_n)^2]/n^2<\infty$, so $\mathbb{P}^*(\frac{1}{n}\sum_{k=1}^nW^*_k\rightarrow w(1))=1$ by \citep[][Lemma 2]{berti2011}. As a result,
$$\frac{N_n(\{1\}^c)}{M_n(\{1\}^c)}=\frac{\theta\nu(\{1\}^c)}{M_n(\{1\}^c)}+\frac{\sum_{k=1}^{M_n(\{1\}^c)-1}W_{T_k}}{M_n(\{1\}^c)}\longrightarrow w(1)\qquad\mbox{a.s.}[\mathbb{P}].$$
\end{proof}

\begin{proof}[\textnormal{\textbf{Proof of \cref{results:dps:mon}}}]~ \\

\noindent\emph{Part I.} ($P_n\overset{w}{\longrightarrow}\delta_1$ a.s.[$\mathbb{P}$], $\hat{P}_n\overset{w}{\longrightarrow}\delta_1$ a.s.[$\mathbb{P}$], $\frac{N_n}{n}\overset{a.s.}{\longrightarrow}w(1)$).
Let $t\in(0,1)$. Take $s\in(t,1)$ and $\lambda\in(0,1)$ as in the statement of \cref{appendix:mon:lemma2}. It follows from \cref{appendix:mon:lemma1,appendix:mon:lemma2} that
$$\limsupN P_n[0,t]\leq\limsupN\frac{N_n[0,t]}{N_n(t,1]}\leq\limsupN\frac{1}{N_n(s,1]^{1-\lambda}}\frac{N_n[0,t]}{N_n(s,1]^\lambda}=0\quad\textnormal{a.s.}[\mathbb{P}];$$
thus, $P_n(t,1]\overset{a.s.}{\longrightarrow}1$ for every $t\in(0,1)$, and so $P_n\overset{w}{\longrightarrow}\delta_1$ a.s.$[\mathbb{P}]$. By \citep[][Lemma 2]{berti2011}, $\hat{P}_n(t,1]\overset{a.s.}{\longrightarrow}1$ for every $t\in(0,1)$ and, similarly, $\hat{P}_n\overset{w}{\longrightarrow}\delta_1$ a.s.$[\mathbb{P}]$. 

On the other hand, \eqref{eq:appendix:dps:expected_weight} and the a.s. weak convergence of $P_n$ to $\delta_1$ imply $\mathbb{E}[W_{n+1}|\mathcal{F}_n]\overset{a.s.}{\longrightarrow}w(1)$. As $\sumINF\mathbb{E}[W_n^2]/n^2\leq\sumINF\beta^2/n^2<\infty$, then, by \citep[][Lemma 2]{berti2011},
\begin{equation}\label{eq:appendix:dps:mon:rate1}
\frac{N_n}{n}\overset{a.s.}{\longrightarrow}w(1).
\end{equation}

\noindent\emph{Part II.} ($P_n(\{1\}^c)\overset{a.s.}{\longrightarrow}\eta$). Suppose $0<\nu(\{1\})<1$; else, if $\nu(\{1\})=0$ (or $\nu(\{1\})=1$), then $P_n(\{1\}^c)=1$ ($P_n(\{1\}^c)=0$), and so $P_n(\{1\}^c)\overset{a.s.}{\longrightarrow}\eta$ with $\eta=1$ ($\eta=0$).\\

\noindent\emph{Case 1: $\nu(t,1)=0$ for some $t\in(0,1)$}.
It follows that $P_n(\{1\}^c)=P_n[0,t]$, so Part I implies $\limsup_n P_n[0,t]\leq\delta_1([0,t])=0$ a.s.$[\mathbb{P}]$, and thus $P_n(\{1\}^c)\overset{a.s.}{\longrightarrow}0$.\\

\noindent\emph{Case 2: $\nu(t,1)>0$ for each $t\in(0,1)$}.
By \cref{appendix:mon:lemma3}, $N_n(\{1\}^c)\bigl/M_n(\{1\}^c)\overset{a.s.}{\longrightarrow} w(1)$ and, similarly, $N_n(\{1\})\bigl/M_n(\{1\})\overset{a.s.}{\longrightarrow}w(1)$. Let $\psi\in(0,1)$. Then
\begingroup\allowdisplaybreaks
\begin{align*}
\mathbb{E}\Bigl[\frac{M_{n+1}(\{1\})^\psi}{M_{n+1}(\{1\}^c)}-\frac{M_n(\{1\})^\psi}{M_n(\{1\}^c)}\Bigr|\mathcal{F}_n\Bigr]&=\frac{M_n(\{1\})^\psi}{M_n(\{1\}^c)}\mathbb{E}\biggl[\Bigl(1+\frac{\delta_{X_{n+1}}(\{1\})}{M_n(\{1\})}\Bigr)^\psi\frac{M_n(\{1\}^c)}{M_{n+1}(\{1\}^c)}-1\Bigr|\mathcal{F}_n\biggr]\\
&\leq\frac{M_n(\{1\})^\psi}{M_n(\{1\}^c)N_n}\Bigl\{\psi\frac{N_n(\{1\})}{M_n(\{1\})}-\frac{N_n(\{1\}^c)}{M_n(\{1\}^c)+1}\Bigr\},
\end{align*}
\endgroup
using that $(1+x)^\psi\leq 1+\psi x$ for $0\leq x\leq 1$. Since
$$\limsupN\Bigl(\psi\frac{N_n(\{1\})}{M_n(\{1\})}-\frac{N_n(\{1\}^c)}{M_n(\{1\}^c)+1}\Bigr)<0\qquad\mbox{a.s.}[\mathbb{P}],$$
we can show, arguing as in the proof of \cref{appendix:mon:lemma2}, that $(M_n(\{1\})^\psi\bigr/M_n(\{1\}^c))_{n\geq 1}$ converges a.s.$[\mathbb{P}]$. But $\psi$ is arbitrary and $0<(1+\psi)/2<1$, so it holds that 
$$\frac{M_n(\{1\})^\psi}{M_n(\{1\}^c)}=M_n(\{1\})^\frac{\psi-1}{2}\frac{M_n(\{1\})^\frac{1+\psi}{2}}{M_n(\{1\}^c)}\overset{a.s.}{\longrightarrow}0.$$
Analogously, $M_n(\{1\}^c)^\psi\bigl/M_n(\{1\})\overset{a.s.}{\longrightarrow}0$. Therefore, for every $\psi<1$,
\begin{equation}\label{eq:appendix:dps:mon:rate2}
\frac{n^\psi}{N_n(\{1\})}=\Bigl(\frac{M_n(\{1\}^c)+M_n(\{1\})-2}{M_n(\{1\})^{1/\psi}}\Bigr)^\psi\frac{M_n(\{1\})}{N_n(\{1\})}\overset{a.s.}{\longrightarrow}0\qquad\mbox{and}\qquad\frac{n^\psi}{N_n(\{1\}^c)}\overset{a.s.}{\longrightarrow}0.
\end{equation}
Define $Z_n=\log\frac{N_n(\{1\}^c)}{N_n(\{1\})}$, $n\geq0$. By \citep[][Lemma 3.2]{pemantle1999}, $(Z_n)_{n\geq 0}$ converges $\mathbb{P}$-a.s. in $[-\infty,\infty)$ if
\begin{equation}\label{eq:appendix:dps:mon:result}
\sum_{n=0}^\infty\mathbb{E}[Z_{n+1}-Z_n|\mathcal{F}_n]<\infty\;\;\textnormal{a.s.}[\mathbb{P}]\qquad\textnormal{and}\qquad\sum_{n=0}^\infty\mathbb{E}\bigl[(Z_{n+1}-Z_n)^2\bigr|\mathcal{F}_n\bigr]<\infty\;\;\textnormal{a.s.}[\mathbb{P}].
\end{equation}
In that case, $(N_n(\{1\}^c)\bigr/N_n(\{1\}))_{n\geq 0}$ converges $\mathbb{P}$-a.s. in $[0,\infty)$ and there exists a random variable $\eta\in[0,1)$ such that
$$P_n(\{1\}^c)=\frac{N_n(\{1\}^c)}{N_n(\{1\}^c)+N_n(\{1\})}\overset{a.s.}{\longrightarrow}\eta.$$
To prove \eqref{eq:appendix:dps:mon:result}, observe that
\begingroup\allowdisplaybreaks
\begin{align*}
\mathbb{E}[Z_{n+1}-Z_n|\mathcal{F}_n]&=\mathbb{E}\biggl[\log\frac{N_{n+1}(\{1\}^c)}{N_n(\{1\}^c)}\delta_{X_{n+1}}(\{1\}^c)-\log\frac{N_{n+1}(\{1\})}{N_n(\{1\})}\delta_{X_{n+1}}(\{1\})\Bigr|\mathcal{F}_n\biggr]\\
&=\mathbb{E}\biggl[\int_0^{W_{n+1}}\frac{\delta_{X_{n+1}}(\{1\}^c)}{N_n(\{1\}^c)+t}dt-\int_0^{W_{n+1}}\frac{\delta_{X_{n+1}}(\{1\})}{N_n(\{1\})+t}dt\Bigr|\mathcal{F}_n\biggr]\\
&\begin{aligned}\leq\mathbb{E}\biggl[\delta_{X_{n+1}}(\{1\}^c)\Bigl(\frac{W_{n+1}}{N_n(\{1\}^c)}&-\frac{W_{n+1}^2}{2N_n(\{1\}^c)^2}+k_1\frac{W_{n+1}^3}{3N_n(\{1\}^c)^3}\Bigr)\Bigr|\mathcal{F}_n\biggr]\\
&-\mathbb{E}\biggl[\delta_{X_{n+1}}(\{1\})\Bigl(\frac{W_{n+1}}{N_n(\{1\})}-k_2\frac{W_{n+1}^2}{2N_n(\{1\})^2}\Bigr)\Bigr|\mathcal{F}_n\biggr]\end{aligned}\\
&\leq\frac{1}{N_n}\Bigl(k_1\frac{\beta^3}{3N_n(\{1\}^c)^2}+k_2\frac{\beta^2}{2N_n(\{1\})}\Bigr),
\end{align*}
\endgroup
where we have used a Taylor expansion of the function $f(x)=1/(x+t)$, $x>0$ with $t\in\mathbb{R}_+$ fixed, for some constants $k_1,k_2\geq0$. Then \eqref{eq:appendix:dps:mon:rate1}-\eqref{eq:appendix:dps:mon:rate2} imply $\sum_{n=0}^\infty\mathbb{E}[Z_{n+1}-Z_n|\mathcal{F}_n]<\infty$ a.s.$[\mathbb{P}]$. On the other hand,
\begingroup\allowdisplaybreaks
\begin{align*}
\mathbb{E}\bigl[(Z_{n+1}-Z_n)^2\bigr|\mathcal{F}_n\bigr]&\leq\mathbb{E}\biggl[\frac{\delta_{X_{n+1}}(\{1\}^c)}{N_n(\{1\}^c)^2}\Bigl(\int_0^{W_{n+1}}dt\Bigr)^2+\frac{\delta_{X_{n+1}}(\{1\})}{N_n(\{1\})^2}\Bigl(\int_0^{W_{n+1}}dt\Bigr)^2\Bigr|\mathcal{F}_n\biggr]\\
&\leq P_n(\{1\}^c)\frac{\beta^2}{N_n(\{1\}^c)^2}+P_n(\{1\})\frac{\beta^2}{N_n(\{1\})^2},
\end{align*}
\endgroup
and thus $\sum_{n=0}^\infty\mathbb{E}[(Z_{n+1}-Z_n)^2|\mathcal{F}_n]<\infty$ a.s.$[\mathbb{P}]$.
\end{proof}

\begin{proof}[\textnormal{\textbf{Proof of \cref{results:dps:arbitrary}}}]
Let us define $g(x)=w(x)/\bar{w}$, $x\in\mathbb{X}$ and, for $n\geq 1$ and $A\in\mathcal{B}[0,1]$,
$$\tilde{X}_n=g(X_n),\qquad \tilde{W}_n=W_n,\qquad \tilde{P}_n(A)=P_n(g^{-1}(A)),\qquad\tilde{\nu}(A)=\nu(g^{-1}(A)).$$
Then $\mathbb{E}[\tilde{W}_n|\tilde{X}_n]=\bar{w}\tilde{X}_n$ and, for every $t\in(0,1)$, by \eqref{eq:def:condition},
$$\tilde{\nu}(t,1]=\nu(\{x\in\mathbb{X}:w(x)>t\bar{w}\})=\nu_w(t\bar{w},\bar{w}]>0;$$
thus, $1\in\mbox{supp}(\tilde{\nu})$. On the other hand, for every $A\in\mathcal{B}[0,1]$,
$$\tilde{P}_n(A)=\sumN\frac{\tilde{W}_i}{\theta+\sum_{j=1}^n\tilde{W}_j}\delta_{\tilde{X}_i}(A)+\frac{\theta}{\theta+\sum_{j=1}^n\tilde{W}_j}\tilde{\nu}(A).$$
Therefore, $(\tilde{X}_n)_{n\geq 1}$ is\footnote{Strictly speaking, $(\tilde{X}_n)_{n\geq 1}$ differs from the definition of a DPS in that $\tilde{W}_n$ is a function of $(X_n,U_n)$ but not of $(\tilde{X}_n,U_n)$, and $\tilde{P}_n$ is the conditional distribution of $\tilde{X}_{n+1}$ given $\mathcal{F}_n$ instead of $\sigma(\tilde{X}_1,U_1,\ldots,\tilde{X}_n,U_n)$. Nevertheless, the conclusions of \cref{results:dps:mon} continue to hold for the process $(\tilde{X}_n)_{n\geq 1}$ since $\mathbb{E}[\tilde{W}_n|\mathcal{F}_{n-1}\vee\sigma(X_n)]=\tilde{w}(\tilde{X}_n)$.} a $[0,1]$-valued DPS with the continuous and strictly increasing expected weight function $\tilde{w}(t)=\bar{w}t$, $t\in[0,1]$. It follows immediately from \cref{results:dps:mon} that
$$\frac{1}{n}\sumN W_i=\frac{1}{n}\sumN \tilde{W}_i\overset{a.s.}{\longrightarrow}\tilde{w}(1)=\bar{w}.$$

\noindent\emph{Part I.} ($P_n(\mathcal{D}_\delta^c)\overset{a.s.}{\longrightarrow}0$). 
Fix $\delta>0$. Let $\epsilon=(\bar{w}-\sup_{x\in \mathcal{D}^c_\delta}w(x))/2$. Then $\epsilon>0$ by \eqref{eq:def:condition}, so \cref{results:dps:mon} implies
$$P_n(\mathcal{D}_\delta^c)\leq\mathbb{P}\bigl(w(X_{n+1})<\bar{w}-\epsilon\,|\mathcal{F}_n\bigr)=\tilde{P}_n[0,1-\epsilon/\bar{w})\overset{a.s.}{\longrightarrow} 0.$$

\noindent\emph{Part II.} ($P_n(\mathcal{D}^c)\overset{a.s.}{\longrightarrow}\eta$). By \cref{results:dps:mon}, there exists a random variable $\eta\in[0,1]$ such that
$$P_n(\mathcal{D}^c)=\tilde{P}_n(\{1\}^c)\overset{a.s.}{\longrightarrow}\eta.$$
If it holds $\bar{w}>\bar{w}^c$, then there exists $t\in(0,1)$ such that $\nu_w(t\bar{w},\bar{w})=0$. But $\tilde{\nu}(t,1)=\nu_w(t\bar{w},\bar{w})=0$, so $P_n(\mathcal{D}^c)=\tilde{P}_n[0,t]\overset{a.s.}{\longrightarrow}0$.\\

\noindent\emph{Part III.} ($\|P_n-\tilde{P}\|\overset{a.s.}{\longrightarrow}0$, $\|\hat{P}_n-\tilde{P}\|\overset{a.s.}{\longrightarrow}0$). Suppose $\bar{w}>\bar{w}^c$. It follows from Part II that $P_n(\mathcal{D})\overset{a.s.}{\longrightarrow}1$, so $\hat{P}_n(\mathcal{D})\overset{a.s.}{\longrightarrow}1$ by \citep[][Lemma 2]{berti2011}, and thus $\mathbb{P}(X_n\in\mathcal{D}\textnormal{ i.o.})=1$ by \cite[][Theorem 1]{dubins1965}.

Let us define $T=0$ and, for $n\geq1$,
$$T_n=\inf\{m\in\mathbb{N}:m>T_{n-1},X_m\in\mathcal{D}\}\qquad\textnormal{and}\qquad P_n(\cdot\,|\mathcal{D})=P_n(\cdot\cap\mathcal{D})\bigr/P_n(\mathcal{D}).$$
Put $\Omega^*=\bigcap_{n=0}^\infty\{T_n<\infty\}$. Each $T_n$ is a $\mathcal{F}$-stopping time such that $T_n<\infty$ a.s.$[\mathbb{P}]$, and $T_n\overset{a.s.}{\longrightarrow}\infty$; thus, $\mathbb{P}(\Omega^*)=1$. Then, arguing as in \cref{appendix:mon:lemma3}, one can show that $P_n^*(\cdot)\equiv P_{T_{n+1}-1}(\cdot\,|\mathcal{D})$ is a version of the conditional distribution of $X_{T_{n+1}}$ given $\sigma(X_{T_1},W_{T_1},\ldots,X_{T_n},W_{T_n})$ on the probability space $(\Omega^*,\mathcal{H}^*,\mathbb{P}^*)\equiv(\Omega^*,\mathcal{H}\cap\Omega^*,\mathbb{P}(\cdot\,|\Omega^*))$. Moreover, the sequence $(X_{T_n})_{n\geq1}$ is a DPS with a constant expected weight function, so there exists by \cref{results:dps:all} a random probability measure $\tilde{P}^*$ on $\mathcal{D}$ such that 
$$\mathbb{P}^*\bigl(\limN\|P_n^*-\tilde{P}^*\|=0\bigr)=1.$$
Define $\tilde{P}(A)(\omega)=\tilde{P}^*(A\cap\mathcal{D})(\omega)$, $\omega\in\Omega^*$ and $\tilde{P}(A)(\omega)=0$, $\omega\notin\Omega^*$, $A\in\mathcal{X}$. Then $\tilde{P}$ is a random probability measure on $\mathbb{X}$ such that $\tilde{P}(\mathcal{D})=1$ a.s.$[\mathbb{P}]$. On the other hand, $P_n(\cdot\,|\mathcal{D})\equiv P_{m_n}^*(\cdot)$ on $\Omega^*$, where $m_n=M_n(\mathcal{D})-1$, so
$$\mathbb{P}\bigl(\limN\|P_n(\cdot\,|\mathcal{D})-\tilde{P}(\cdot)\|=0\bigr)=\mathbb{P}^*\bigl(\limN\|P_{m_n}^*-\tilde{P}^*\|=0\bigr)=1.$$
Therefore,
$$\|P_n-\tilde{P}\|\leq2\,P_n(\mathcal{D}^c)+\|P_n(\cdot\,|\mathcal{D})-\tilde{P}(\cdot)\|\overset{a.s.}{\longrightarrow}0.$$
Finally, denote by $\hat{P}_n(\cdot\,|\mathcal{D})=\hat{P}_n(\cdot\cap\mathcal{D})\bigr/\hat{P}_n(\mathcal{D})$, $n\geq0$ the relative frequency of $(X_n)_{n\geq1}$ restricted to $\mathcal{D}$, which is well-defined for large $n$. Then, arguing as before, we have
$$\|\hat{P}_n-\tilde{P}\|\leq2\,\hat{P}_n(\mathcal{D}^c)+\|\hat{P}_n(\cdot\,|\mathcal{D})-\tilde{P}(\cdot)\|\overset{a.s.}{\longrightarrow}0.$$
\end{proof}

\begin{proof}[\textnormal{\textbf{Proof of \cref{results:dps:arbitrary:number}}}]
Define $\theta_n=\theta/N_n$ and $Z_n=L_n-L_{n-1}$, $n\geq 1$, with $L_0=0$ and $\theta_0=1$. As $\nu$ is diffuse and $Z_n\in\{0,1\}$, then $L_n=\sumN\mathbbm{1}_{\{Z_i=1\}}$ and $\theta_n=\mathbb{P}(Z_{n+1}=1|\mathcal{F}_n)$. Moreover, $\theta_n\geq\theta/(\theta+\beta n)$, so $\sumINF\mathbb{P}(Z_n=1|\mathcal{F}_{n-1})=\infty$ a.s.$[\mathbb{P}]$ and, by \citep[][Theorem 1]{dubins1965},
$$\frac{L_n}{\sum_{k=1}^n\theta_{k-1}}=\frac{\sumN\mathbbm{1}_{\{Z_i=1\}}}{\sum_{k=1}^n\mathbb{P}(Z_k=1|\mathcal{F}_{k-1})}\overset{a.s.}{\longrightarrow}1.$$
As $N_n/n\overset{a.s.}{\longrightarrow}\bar{w}$ by \cref{results:dps:arbitrary}, we obtain
$$\frac{L_n}{\log n}=\frac{L_n}{\sum_{k=1}^n\theta_{k-1}}\biggl(\frac{1}{\log n}+\frac{\theta}{\log n}\sum_{k=1}^{n-1}\frac{1}{k}\Bigl(\frac{k}{N_k}\Bigr)\biggr)\overset{a.s.}{\longrightarrow}\frac{\theta}{\bar{w}},$$
using the fact that, for any $(a_n)_{n\geq1}\subseteq\mathbb{R}_+:a_n\rightarrow a$, it holds $\frac{1}{\log n}\sum_{k=1}^nk^{-1}a_k\rightarrow a$.
\end{proof}

\begin{proof}[\textnormal{\textbf{Proof of \cref{results:clt:arbitrary:rate}}}]
If $\nu(\mathcal{D}^c)=0$, the result is immediate with $\xi=0$. Suppose $\nu(\mathcal{D}^c)>0$. Then $\bar{w}^c\in\mbox{supp}((\nu_{|\mathcal{D}^c})_w)$ from the definition of $\bar{w}^c$. Moreover, $\nu(\{x\in\mathcal{D}^c:w(x)>\bar{w}^c\})=0$; otherwise, $(\bar{w}^c,\bar{w})\cap\mbox{supp}((\nu_{|\mathcal{D}^c})_w)\neq\emptyset$ and there exists $u>\bar{w}^c$ such that $u\in\mbox{supp}((\nu_{|\mathcal{D}^c})_w)$, absurd. As a result, $\bar{w}^c$ plays the same role for the subsequence of non-dominant observations as $\bar{w}$ does for $(X_n)_{n\geq 1}$. Therefore, by \cref{results:dps:arbitrary}, $N_n(\mathcal{D}^c)\bigl/M_n(\mathcal{D}^c)\overset{a.s.}{\longrightarrow}\bar{w}^c$. Proceeding as in Part II of the proof of \cref{results:dps:mon}, we can show that, for every $\psi<\bar{w}^c/\bar{w}$,
$$\frac{M_n(\mathcal{D})^\psi}{M_n(\mathcal{D}^c)}\overset{a.s.}{\longrightarrow}0,\qquad\mbox{and then}\qquad\frac{M_n(\mathcal{D}^c)}{n^\psi}\overset{a.s.}{\longrightarrow}\infty.$$
In addition, $\log\frac{N_n(\mathcal{D}^c)}{N_n(\mathcal{D})^{\bar{w}^c/\bar{w}}}$ converges $\mathbb{P}$-a.s. in $[-\infty,\infty)$, so there exists a random variable $\xi\in[0,\infty)$ such that, letting $\gamma=1-\bar{w}^c/\bar{w}$,
\begingroup\allowdisplaybreaks
\begin{gather*}
n^\gamma\cdot\hat{P}_n(\mathcal{D}^c)=\frac{M_n(\mathcal{D}^c)-1}{N_n(\mathcal{D}^c)}\Bigl(\frac{N_n}{n}\frac{N_n(\mathcal{D}^c)^{\bar{w}/\bar{w}^c}}{N_n(\mathcal{D})+N_n(\mathcal{D}^c)}\Bigr)^{\bar{w}^c/\bar{w}}\overset{a.s.}{\longrightarrow}\xi,\\
n^\gamma\cdot P_n(\mathcal{D}^c)=n^\gamma\cdot\Bigl(\hat{P}_n(\mathcal{D}^c)-\frac{1}{n}\Bigr)\frac{n}{N_n}\frac{N_n(\mathcal{D}^c)}{M_n(\mathcal{D}^c)}\overset{a.s.}{\longrightarrow}\frac{\bar{w}^c}{\bar{w}}\xi.
\end{gather*}
\endgroup
\end{proof}

\begin{proof}[\textnormal{\textbf{Proof of \cref{results:clt:arbitrary:discont}}}]
By \cref{results:dps:arbitrary}, $P_n(\mathcal{D})\overset{a.s.}{\longrightarrow}1$, so we can proceed as in \cref{results:dps:arbitrary} to derive results for the subsequence of dominant observations. Let us define $T_0=0$ and, for $n\geq1$
$$T_n=\inf\{n>T_{n-1}:X_n\in\mathcal{D}\},\qquad P_n(\cdot\,|\mathcal{D})=\frac{P_n(\cdot\cap\mathcal{D})}{P_n(\mathcal{D})},\qquad\hat{P}_n(\cdot\,|\mathcal{D})=\frac{\hat{P}_n(\cdot\cap\mathcal{D})}{\hat{P}_n(\mathcal{D})},$$
which are well-defined for large $n$. We further let $m_n=M_n(\mathcal{D})-1$ and $\mathcal{G}=(\mathcal{G}_n)_{n\geq0}$ be the filtration on $\mathcal{H}$, given by $\mathcal{G}_0=\{\emptyset,\Omega\}$ and $\mathcal{G}_n=\sigma(X_{T_1},U_{T_1},\ldots,X_{T_{m_n}},U_{T_{m_n}})$, $n\geq1$. With a slight abuse of notation, $(X_{T_{m_n}})_{n\geq 1}$ forms a subsequence of dominant observations,
$$P_n(\cdot\,|\mathcal{D})=\frac{\theta\nu(\cdot\cap\mathcal{D})+\sum_{k=1}^{m_n}W_{T_k}\delta_{X_{T_k}}(\cdot)}{\theta\nu(\mathcal{D})+\sum_{k=1}^{m_n}W_{T_k}},$$
is a version of the conditional distribution of $X_{T_{m_{n+1}}}$ given $\mathcal{G}_n$ (see the proof of \cref{results:dps:arbitrary}), and $(\hat{P}_n(\cdot\,|\mathcal{D}))_{n\geq1}$ is the relative frequency of $(X_{T_{m_n}})_{n\geq 1}$ restricted to $\mathcal{D}$. Observe that $P_n(\cdot\,|\mathcal{D})$ and $\hat{P}_n(\cdot\,|\mathcal{D})$ remain unchanged unless $m_n>m_{n-1}$. It follows from \cref{results:dps:all} that, for every $A\in\mathcal{X}$,
$$P_n(A|\mathcal{D})\overset{a.s.}{\longrightarrow}\tilde{P}(A)\qquad\mbox{and}\qquad\hat{P}_n(A|\mathcal{D})\overset{a.s.}{\longrightarrow}\tilde{P}(A).$$

We can prove the convergence to Gaussian limits of 
\begin{equation}\label{eq:appendix:clt:arbitrary:difference}
\sqrt{m_n}\bigl(\hat{P}_n(A|\mathcal{D})-P_n(A|\mathcal{D})\bigr)\qquad\mbox{and}\qquad\sqrt{m_n}\bigl(P_n(A|\mathcal{D})-\tilde{P}(A)\bigr),
\end{equation}
by applying \citep[][Theorem 1]{berti2011} and \citep[][Proposition 1]{berti2011}, respectively. Indeed, one can easily show that \eqref{eq:appendix:clt:arbitrary:difference} satisfy the assumptions of \citep[][Theorem 1, Proposition 1]{berti2011} using techniques from \citep[][Theorem 4]{berti2010} and \citep[][Corollary 3]{berti2011}. As a result,
\begin{gather*}
\sqrt{m_n}\bigl(\hat{P}_n(A|\mathcal{D})-P_n(A|\mathcal{D})\bigr)\overset{stably}{\longrightarrow}\mathcal{N}\bigl(0,U(A)\bigr),\\
\sqrt{m_n}\bigl(P_n(A|\mathcal{D})-\tilde{P}(A)\bigr)\overset{a.s.cond.}{\longrightarrow}\mathcal{N}\bigl(0,V(A)\bigr)\quad\textnormal{w.r.t. }\mathcal{G}.
\end{gather*}  
Since $M_n(\mathcal{D})/n\overset{a.s.}{\longrightarrow}1$, a generalized Slutsky's theorem (see also \citep[][Theorem 6]{fortini2020}) implies
\begin{gather*}
\sqrt{n}\bigl(\hat{P}_n(A|\mathcal{D})-P_n(A|\mathcal{D})\bigr)\overset{stably}{\longrightarrow}\mathcal{N}\bigl(0,U(A)\bigr),\\
\sqrt{n}\bigl(P_n(A|\mathcal{D})-\tilde{P}(A)\bigr)\overset{a.s.cond.}{\longrightarrow}\mathcal{N}\bigl(0,V(A)\bigr)\quad\textnormal{w.r.t. }\mathcal{G}.
\end{gather*}
But $\tilde{P}(A)$ is $\sigma(X_{T_{m_1}},X_{T_{m_2}},\ldots)$-measurable from $\hat{P}_n(A|\mathcal{D})\overset{a.s.}{\longrightarrow}\tilde{P}(A)$ and, using \eqref{eq:def:weights}, we can show that $(X_{T_{m_{n+1}}},X_{T_{m_{n+2}}},\ldots)$ and $\mathcal{F}_n$ are conditionally independent given $\mathcal{G}_n$; therefore,
$$\sqrt{n}\bigl(P_n(A|\mathcal{D})-\tilde{P}(A)\bigr)\overset{a.s.cond.}{\longrightarrow}\mathcal{N}\bigl(0,V(A)\bigr)\quad\textnormal{w.r.t. }\mathcal{F}.$$
On the other hand, \cref{results:dps:mon} implies that
$$\bigl|\sqrt{n}(P_n(A|\mathcal{D})-D_n(A)\bigr|=\sqrt{n}|P_n(A|\mathcal{D})-P_n(A)|\leq2\sqrt{n}\cdot P_n(\mathcal{D}^c)\overset{a.s.}{\longrightarrow} 0.$$
Then, using again the generalized Slutsky's theorem, $D_n(A)\overset{a.s.cond.}{\longrightarrow}\mathcal{N}\bigl(0,V(A)\bigr)$ w.r.t. $\mathcal{F}$.

Regarding the second result, we have
$$\bigl|\sqrt{n}(\hat{P}_n(A|\mathcal{D})-P_n(A|\mathcal{D}))-C_n(A)\bigr|\leq2\sqrt{n}\cdot P_n(\mathcal{D}^c)+2\sqrt{n}\cdot\hat{P}_n(\mathcal{D}^c)\overset{a.s.}{\longrightarrow}0;$$
therefore, $C_n(A)\overset{stably}{\longrightarrow}\mathcal{N}\bigl(0,U(A)\bigr)$.
\end{proof}

\subsection*{Acknowledgments}

The authors would like to thank the two anonymous referees for their careful reading and valuable comments. H. Sariev was partially supported by the Bulgarian Ministry of Education and Science under the National Research Programme ``Young scientists and postdoctoral students'' approved by DCM No. 577/17.08.2018.

\bibliography{mybib_submit} 
\bibliographystyle{plain}

\end{document}